\documentclass{amsart}
\usepackage{amssymb, amsmath, latexsym, mathrsfs, verbatim, url, bbm, marvosym}

\theoremstyle{plain}
\newtheorem{theorem}{Theorem}
\newtheorem{lemma}[theorem]{Lemma}

\newtheorem{corollary}[theorem]{Corollary}

\theoremstyle{definition}

\newtheorem{question}[theorem]{Question}

\newcommand{\dom}{\mathsf{dom}}
\newcommand{\ran}{\mathsf{ran}}
\newcommand{\add}{\mathsf{add(meager)}}
\newcommand{\id}{\mathsf{id}}
\newcommand{\di}{\mathsf{d}}
\newcommand{\MAsigma}{\mathsf{MA(\sigma\textrm{-}centered)}}

\newcommand{\ZFC}{\mathsf{ZFC}}

\newcommand{\Gd}{\mathsf{G_\delta}}

\newcommand{\Ss}{\mathcal{S}}
\newcommand{\CC}{\mathcal{C}}
\newcommand{\DD}{\mathcal{D}}
\newcommand{\TT}{\mathcal{T}}

\newcommand{\GG}{\mathcal{G}}

\newcommand{\HH}{\mathcal{H}}

\newcommand{\VV}{\mathcal{V}}

\newcommand{\ZZZ}{\mathbb{Z}}
\newcommand{\QQQ}{\mathbb{Q}}
\newcommand{\RRR}{\mathbb{R}}
\newcommand{\cccc}{\mathfrak{c}}

\begin{document}

\title{A homogeneous space whose complement is rigid}

\author{Andrea Medini}
\address{Kurt G\"odel Research Center for Mathematical Logic
\newline\indent University of Vienna
\newline\indent W\"ahringer Stra{\ss}e 25
\newline\indent A-1090 Wien, Austria}
\email{andrea.medini@univie.ac.at}
\urladdr{http://www.logic.univie.ac.at/\~{}medinia2/}

\author{Jan van Mill}
\address{KdV Institute for Mathematics
\newline\indent University of Amsterdam
\newline\indent Science Park 904
\newline\indent P.O. Box 94248
\newline\indent 1090 GE Amsterdam, The Netherlands}
\email{j.vanMill@uva.nl}
\urladdr{http://staff.fnwi.uva.nl/j.vanmill/}

\author{Lyubomyr Zdomskyy}
\address{Kurt G\"odel Research Center for Mathematical Logic
\newline\indent University of Vienna
\newline\indent W\"ahringer Stra{\ss}e 25
\newline\indent A-1090 Wien, Austria}
\email{lyubomyr.zdomskyy@univie.ac.at}
\urladdr{http://www.logic.univie.ac.at/\~{}lzdomsky/}


\thanks{The first-listed and third-listed authors were supported by the FWF grant I 1209-N25. The second-listed author acknowledges generous hospitality and support from the Kurt G\"odel Research Center for Mathematical Logic. The third-listed author also thanks the Austrian Academy of Sciences for its generous support through the APART Program.}

\date{October 2, 2014}

\begin{abstract}
We construct a homogeneous subspace of $2^\omega$ whose complement is dense in $2^\omega$ and rigid. Using the same method, assuming Martin's Axiom, we also construct a countable dense homogeneous subspace of $2^\omega$ whose complement is dense in $2^\omega$ and rigid.
\end{abstract}

\maketitle

All spaces are assumed to be separable and metrizable. Our reference for general topology will be \cite{vanmilli}. Given a space $X$, we will denote by $\HH(X)$ the group of homeomorphisms of $X$ with the operation of composition. Recall that a space $X$ is \emph{rigid} if $\HH(X)=\{\id\}$. Every space of size at most $1$ is a trivial example of rigid space. We refer the reader to \cite[Section 13]{vandouwen} for the early history of rigid spaces. Let us only mention that the first non-trivial example of rigid space was the zero-dimensional subspace of $\RRR$ constructed by Kuratowski in \cite{kuratowski}.

Recall the following definitions. A space $X$ is \emph{homogeneous} if for every pair $(x,y)$ of points of $X$ there exists $h\in\HH(X)$ such that $h(x)=y$. A space $X$ is \emph{countable dense homogeneous} if for every pair $(A,B)$ of countable dense subsets of $X$ there exists $h\in\HH(X)$ such that $h[A]=B$. These are classical notions, and they have been studied in depth (see for example the survey \cite{arkhangelskiivanmill}).

Examples of rigid spaces abound in the literature. See \cite{degrootwille} for examples of rigid continua of positive finite dimension. Other examples of this type were given in \cite{ancelsingh} and \cite{ancelduvallsingh}, with the additional property that their square is a manifold, hence homogeneous and countable dense homogeneous (see \cite[Theorem 1.6.9 and Corollary 1.6.8]{vanmilli}). For a non-trivial ``very'' rigid continuum, see \cite{cook}. A rigid space whose square is homeomorphic to the Hilbert cube $[0,1]^\omega$, hence homogeneous and countable dense homogeneous (see \cite[Theorem 1.6.6 and Theorem 1.6.9]{vanmilli}), was given in \cite{vanmillr}. For other infinite-dimensional examples, see \cite{dijkstrac} and \cite{dijkstras}. A non-trivial zero-dimensional rigid space whose square is homogeneous was given in \cite{lawrence}.\footnote{\,On the other hand, the existence of a non-trivial zero-dimensional rigid space whose square is countable dense homogeneous is an open problem (see \cite[Question 1.11]{medini}).} However, by \cite{vanengelenmillersteel}, there are no non-trivial rigid zero-dimensional Borel spaces.

The interest of most of the above examples lies in the fact that rigid spaces have as few homeomorphisms as possible, while homogeneous spaces and countable dense homogeneous spaces must have ``many'' homeomorphisms. In this article, we will show how to obtain both extremes simultaneously. More precisely, we will construct subspaces of $2^\omega$ with a prescribed homogeneity-type property, while making sure at the same time that their complements are dense in $2^\omega$ and rigid. Our main results are Theorem \ref{homogeneous} and Theorem \ref{cdh}, whose proofs are both based on the general method given by Theorem \ref{general}. For other examples of a similar flavor, see \cite{vanengelen}, \cite{vanengelenvanmill}, \cite{vanmillwattel} and \cite{shelah}.

We will say that a subspace $X$ of $2^\omega$ is \emph{relatively homogeneous} if for every pair $(x,y)$ of points of $X$ there exists $h\in\HH(2^\omega)$ such that $h[X]=X$ and $h(x)=y$. Every relatively homogeneous subspace of $2^\omega$ is obviously homogeneous, and Corollary \ref{hnotrh} gives a $\ZFC$ counterexample to the reverse implication. Similarly, we will say that a subspace $X$ of $2^\omega$ is \emph{relatively countable dense homogeneous} if for every pair $(A,B)$ of countable dense subsets of $X$ there exists $h\in\HH(2^\omega)$ such that $h[X]=X$ and $h[A]=B$. Every relatively countable dense homogeneous subspace of $2^\omega$ is obviously countable dense homogeneous, and Corollary \ref{cdhnotrcdh} shows that the reverse implication is not provable in $\ZFC$. This answers a question raised by the referee of the recent paper \cite{kunenmedinizdomskyy}, which was the original motivation for the research contained in this article.

\section{Preliminaries}

Recall that a space is \emph{crowded} if it is non-empty and has no isolated points. We will write $X\approx Y$ to mean that the spaces $X$ and $Y$ are homeomorphic. Given a subset $X$ of a space $Z$ and a subgroup $\HH$ of $\HH(Z)$, we will let
$$
\HH[X]=\bigcup_{h\in\HH}h[X]
$$
be the closure of $X$ under the action of $\HH$. Notice that $\HH[\HH[X]]=\HH[X]$. For simplicity, we will let $\HH(x)=\HH[\{x\}]$. Furthermore, given a group $\HH$ and a subset $\Ss$ of $\HH$, we will denote by $\langle\Ss\rangle$ the subgroup of $\HH$ generated by $\Ss$.

Given a surjection $\pi:X\longrightarrow Y$, we will say that a subset $S$ of $X$ is \emph{saturated} with respect to $\pi$ if $\pi^{-1}[\pi[S]]=S$. The proof of the following simple lemma is left to the reader.
\begin{lemma}\label{pi}
Let $\pi:X\longrightarrow Y$ be a continuous surjection between compact spaces. If $A\subseteq X$ is saturated with respect to $\pi$ then $\pi\upharpoonright A:A\longrightarrow\pi[A]$ is a closed continuous surjection.
\end{lemma}

The following lemma, which originally appeared as \cite[Theorem 2.3]{vanmillh} in a slightly different form, will be the key to making our example homogeneous. We need some notation, which will be used throughout the entire paper. Let
$$
Q=\{x\in 2^\omega:\text{there exists }m\in\omega\text{ such that }x_n=x_m\text{ whenever }n\geq m\}.
$$
Given $q\in Q$, let $h_q\in\HH(2^\omega)$ be the homeomorphism defined by $h_q(x)=x+q$ for $x\in2^\omega$, where $+$ denotes addition modulo $2$. Let $\VV=\{h_q:q\in Q\}$, and observe that $\VV$ is a subgroup of $\HH(2^\omega)$.

\begin{lemma}\label{jan}
Assume that $X$ is a subspace of $2^\omega$ such that $\VV[X]=X$. Then $X$ is homogeneous.
\end{lemma}
\begin{proof}
Fix the usual metric $\di$ on $2^\omega$ defined by $\di(x,y)=\sum_{n\in\omega}(|x_n-y_n|/2^n)$. By \cite[Lemma 2.1]{vanmillh} (see also \cite[Corollary 1.9.2]{vanmilli}), it will be enough to show that if $x,y\in X$ then $x$ and $y$ have arbitrarily small homeomorphic neighborhoods. This is straightforward, using the fact that each $h_q$ is an isometry with respect to $\di$.
\end{proof}

The following classical result, which is a well-known tool for ``killing'' homeomorphisms (see \cite{vanmills} for other applications), will be the key to obtaining rigid spaces. For a proof of Lemma \ref{lav},
see \cite[Theorem 3.9 and Exercise 3.10]{kechris}.

\begin{lemma}[Lavrentiev]\label{lav}
Let $f:W\longrightarrow W$ be a homeomorphism, where $W$ is a subspace of some Polish space $Z$. Then there exists a $\Gd$ subset $T$ of $Z$ and a homeomorphism $g:T\longrightarrow T$ such that $f\subseteq g$.
\end{lemma}

Our reference for set theory will be \cite{kunen}. We will denote by $\MAsigma$ the statement that Martin's Axiom for $\sigma$-centered posets holds. Recall that $\add$ is the minimum cardinal $\kappa$ such that there exists a collection $\CC$ of size $\kappa$ consisting of meager subsets of $2^\omega$ such that $\bigcup\CC$ is non-meager in $2^\omega$.\footnote{\,In \cite[Definitions III.1.2 and III.1.6]{kunen}, Kunen uses $\RRR$ instead of $2^\omega$. Since $\RRR\setminus\QQQ\approx\omega^\omega\approx 2^\omega\setminus Q$, it is easy to see that this makes no difference.} It is clear that $\omega_1\leq\add\leq\cccc$. Furthermore, it is well-known that $\MAsigma$ implies $\add=\cccc$ (see for example \cite[Lemmas III.3.22, III.3.26 and III.1.25]{kunen}).

The following lemma, which first appeared as \cite[Lemma 3.2]{baldwinbeaudoin}, will be the key to obtaining our countable dense homogeneous example. See \cite[Corollary 2.2]{medini} for a simpler version of the proof. Given an infinite cardinal $\lambda\leq\cccc$ and a subset $D$ of $2^\omega$, we will say that $D$ is \emph{$\lambda$-dense} if $|D\cap U|=\lambda$ for every non-empty open subset $U$ of $2^\omega$. 
\begin{lemma}[Baldwin, Beaudoin]\label{bb}
Assume $\MAsigma$. Let $\kappa<\cccc$ be a cardinal. Suppose that $A_\alpha$ and $B_\alpha$ are $\lambda_\alpha$-dense subsets of $2^\omega$ for $\alpha < \kappa$, where each $\lambda_\alpha<\cccc$ is an infinite cardinal. Also assume that $A_\alpha\cap A_\beta=\varnothing$ and $B_\alpha\cap B_\beta=\varnothing$ whenever $\alpha\neq\beta$. Then there exists $f\in\HH(2^\omega)$ such that $f[A_\alpha]=B_\alpha$ for every $\alpha < \kappa$.
\end{lemma}

\section{The general method}

The following theorem gives a general method for embedding suitable zero-dimensional spaces into $2^\omega$, so that their complement will be non-trivial and rigid. The strategy of its proof is to combine Lemma \ref{lav} with the idea of ``splitting points'' in a linearly ordered space, which dates back to the classical double arrow space of Alexandroff and Urysohn (see \cite{alexandroffurysohn}).

\begin{theorem}\label{general}
Assume that the following requirements are satisfied.
\begin{itemize}
\item $X$ is a subspace of $2^\omega$.
\item $Y=2^\omega\setminus X$.
\item $\HH$ is a subgroup of $\HH(2^\omega)$.
\item $D$ is a countable dense subset of $2^\omega$ such that $D\subseteq Y$ and $D\cap Q=\varnothing$.
\item $\GG$ is the collection of all homeomorphisms $g$ such that $\dom(g)$ and $\ran(g)$ are $\Gd$ subspaces of $2^\omega$.
\end{itemize}
Furthermore, assume that the following conditions hold.
\begin{enumerate}
\item\label{preserve} $\HH[X]=X$.
\item\label{pushaway} If $h\in\HH\setminus\{\id\}$ then $h[D]\cap D=\varnothing$.
\item\label{kill} If $g\in\GG$ and $|\{x\in\dom(g):g(x)\notin\Ss(x)\}|=\cccc$ for every subgroup $\Ss$ of $\HH$ such that $|\Ss|<\add$ then there exists $z\in\dom(g)\cap X$ such that $g(z)\notin X$.
\end{enumerate}
Then there exists a subspace $X^\ast\approx X$ of $2^\omega$ such that  $2^\omega\setminus X^\ast$ is dense in $2^\omega$ and rigid.
\end{theorem}
\begin{proof}
Let $Z=2^\omega$. First, we will show that $Y\cap U$ is uncountable for every non-empty open subset $U$ of $Z$. Condition (\ref{pushaway}) implies that $h[D]\cap g[D]=\varnothing$ whenever $h,g\in\HH$ and $h\neq g$. In particular, since $\HH[D]\subseteq\HH[Y]=Y$ by condition (\ref{preserve}), the desired conclusion holds if $\HH$ is uncountable. Now assume that $\HH$ is countable. We will actually show that $X$ is a Bernstein set.\footnote{\,Recall that a subset $X$ of $2^\omega$ is a \emph{Bernstein set} if $X\cap K\neq\varnothing$ and $(2^\omega\setminus X)\cap K\neq\varnothing$ for every perfect subset $K$ of $2^\omega$. Notice that the complement of a Bernstein set is also a Bernstein set. Since $2^\omega\approx 2^\omega\times 2^\omega$, every Bernstein set is $\cccc$-dense in $2^\omega$.} Fix a nowhere dense perfect subset $K$ of $Z$. Notice that $Z\setminus\HH[K]$ is a comeager subset of $Z$, hence it contains a perfect subset $K'$. Fix a homeomorphism $g:K\longrightarrow K'$, and notice that $g\in\GG$. Since $|\{x\in\dom(g):g(x)\notin\HH(x)\}|=|K|=\cccc$ and $|\HH|=\omega<\add$, condition (\ref{kill}) shows that $X\cap K\neq\varnothing$. The same reasoning, applied to $g^{-1}$, shows that $Y\cap K\neq\varnothing$.

Let $D^\ast=\{d^-:d\in D\}\cup\{d^+:d\in D\}$, where we use the notation $d^-=\langle d,-1\rangle$ and $d^+=\langle d,1\rangle$ for $d\in D$. Define
$$
Z^\ast=(Z\setminus D)\cup D^\ast.
$$
Consider the function $\pi:Z^\ast\longrightarrow Z$ defined by the following two conditions.
\begin{itemize}
\item $\pi(d^-)=\pi(d^+)=d$ for every $d\in D$.
\item $\pi\upharpoonright (Z\setminus D)=\id$.
\end{itemize}
Let $\prec$ denote the linear ordering on $Z^\ast$ defined by the following two conditions.
\begin{itemize}
\item $d^-$ is the immediate predecessor of $d^+$ for every $d\in D$.
\item $x\prec y$ whenever $x,y\in Z^\ast$ and $\pi(x)<\pi(y)$, where $<$ denotes the usual lexicographic order on $Z$.
\end{itemize}
Consider the order topology induced by $\prec$ on $Z^\ast$. Since $D\cap Q=\varnothing$, it is easy to check that $Z^\ast$ is a compact crowded zero-dimensional space with no isolated points. Therefore $Z^\ast\approx 2^\omega$. Furthermore, it is easy to check that $\pi$ is a continuous surjection. Let $X^\ast$ be the subspace of $Z^\ast$ whose underlying set is $X$, and notice that $X^\ast$ is saturated with respect to $\pi$. Observe that $\pi\upharpoonright X^\ast:X^\ast\longrightarrow X$ is a closed continuous surjection by Lemma \ref{pi}. Since it is also injective (in fact, it is the identity), it follows that $X^\ast\approx X$.

Let $Y^\ast= Z^\ast\setminus X^\ast$, and notice that $Y^\ast\supseteq D^\ast$ is dense in $Z^\ast$. Assume that $f^\ast:Y^\ast\longrightarrow Y^\ast$ is a homeomorphism. Let $B^\ast=\bigcup_{k\in\ZZZ}(f^\ast)^k[D^\ast]$ and $B=\pi[B^\ast]$. Let $W=Y^\ast\setminus B^\ast=Y\setminus B\subseteq Z$, and notice that $f=f^\ast\upharpoonright W:W\longrightarrow W$ is a homeomorphism. By Lemma \ref{lav}, there exists a $\Gd$ subspace $T$ of $Z$ and a homeomorphism $g:T\longrightarrow T$ such that $f\subseteq g$. Notice that $g\in\GG$. Furthermore, since $T'=T\setminus\bigcup_{k\in\ZZZ}g^k[B]\supseteq W$ is still a $\Gd$ and $g'=g\upharpoonright T':T'\longrightarrow T'$ is still a homeomorphism, we can assume without loss of generality that $T\cap B=\varnothing$.

First assume that $|\{x\in T:g(x)\notin\Ss(x)\}|<\cccc$ for some subgroup $\Ss$ of $\HH$ such that $|\Ss|<\add$, and fix such a subgroup. Notice that $T$ is a crowded Polish space because it is a dense $\Gd$ subset of $Z$. Let $T_h=\{x\in T:g(x)=h(x)\}$ for $h\in\Ss$, and notice that each $T_h$ is closed in $T$. Since $|\Ss|<\add$, it follows that at least one $T_h$ has non-empty interior in $T$. Assume, in order to get a contradiction, that $T_h$ has non-empty interior for some $h\in\Ss\setminus\{\id\}$, and fix such an $h$. Fix $a,b\in Z$ such that $a<b$ and $\varnothing\neq (a,b)\cap T\subseteq T_h$.

Fix $d\in D\cap (a,b)$. Since $D\cap Q=\varnothing$ and $Y\cap U$ is uncountable for every non-empty open subset $U$ of $Z$, it is possible to fix a sequence $\langle a_n:n\in\omega\rangle$ consisting of elements of $(a,d)\cap W=(a,d)\cap (Y\setminus B)$ and a sequence $\langle b_n:n\in\omega\rangle$ consisting of elements of $(d,b)\cap W=(d,b)\cap (Y\setminus B)$ such that $a_n\to d$ and $b_n\to d$ in $Z$. Observe that $a_n\to d^-$ and $b_n\to d^+$ in $Y^\ast$. In particular, since $f^\ast$ is a homeomorphism, the sequences $\langle f^\ast(a_n):n\in\omega\rangle$ and $\langle f^\ast(b_n):n\in\omega\rangle$ should converge to different limits in $Y^\ast$, hence in $Z^\ast$.

Notice that $h(a_n)\to h(d)$ and $h(b_n)\to h(d)$ in $Z$ by the continuity of $h$. Since $h\neq\id$, condition (\ref{pushaway}) guarantees that $h(d)\notin D$. Furthermore, the fact that each $a_n\in (a,b)\cap W\subseteq (a,b)\cap T\subseteq T_h$ implies that
$h(a_n)=g(a_n)=f(a_n)=f^\ast(a_n)\notin D$. Therefore $f^\ast(a_n)=h(a_n)\to h(d)$ in $Z^\ast$ as well. Using a similar argument, one sees that $f^\ast(b_n)=h(b_n)\to h(d)$ in $Z^\ast$, which is a contradiction. In conclusion, $T_h$ has non-empty interior if and only if $h=\id$. As one can easily check, this implies that $T_\id$ is dense in $T$. Therefore, the function $g$ is the identity on $T$, which implies that $f^\ast$ is the identity on $Y^\ast$. This shows that $Y^\ast$ is rigid.

Now assume that $|\{x\in T:g(x)\notin\Ss(x)\}|=\cccc$ for every subgroup $\Ss$ of $\HH$ such that $|\Ss|<\add$. Then, by condition (\ref{kill}), we can fix $z\in T\cap X$ such that $g(z)\notin X$. Since $T\cap B=\varnothing$, it is clear that $T=\dom(g)=\ran(g)$ is the disjoint union of $T\cap X$ and $W$. It follows that $g(z)\in W$. This is a contradiction, because it implies that $z=g^{-1}(g(z))=f^{-1}(g(z))\in W$.
\end{proof}

\section{The homogeneous example}

\begin{theorem}\label{homogeneous}
There exists a homogeneous subspace of $2^\omega$ whose complement is dense in $2^\omega$ and rigid.
\end{theorem}
\begin{proof}
Our plan is to apply Theorem \ref{general} with $\HH=\VV$. Let $\GG=\{g_\alpha:\alpha\in\cccc\}$ be an enumeration. Using the fact that $\VV$ is countable, it is easy to construct a countable dense subset $D$ of $2^\omega$ such that $D\cap Q=\varnothing$ and $h[D]\cap D=\varnothing$ for every $h\in\VV\setminus\{\id\}$.

Using transfinite recursion, we will construct increasing sequences $\langle X_\alpha:\alpha\in\cccc\rangle$ and $\langle Y_\alpha:\alpha\in\cccc\rangle$ consisting of subsets of $2^\omega$ so that the following conditions are satisfied for every $\alpha\in\cccc$.
\begin{enumerate}
\item[(I)] $|X_\alpha|<\cccc$ and $|Y_\alpha|<\cccc$.
\item[(II)] $X_\alpha\cap Y_\alpha=\varnothing$.
\item[(III)] $\VV[X_\alpha]=X_\alpha$.
\item[(IV)] If $|\{x\in\dom(g_\alpha):g_\alpha(x)\notin\VV(x)\}|=\cccc$ then there exists $z\in\dom(g_\alpha)\cap X_{\alpha+1}$ such that $g(z)\in Y_{\alpha+1}$.
\end{enumerate}
Start by setting $X_0=\varnothing$ and $Y_0=D$. Take unions at limit stages. At a successor stage $\alpha+1$, suppose that $X_\alpha$ and $Y_\alpha$ have already been constructed. Assume that $|\{x\in\dom(g_\alpha):g_\alpha(x)\notin\VV(x)\}|=\cccc$. Then, it is possible to fix
$$
z\in\dom(g_\alpha)\setminus(\VV[Y_\alpha]\cup g_\alpha^{-1}[X_\alpha])
$$
such that $g_\alpha(z)\notin\VV(z)$. Set $X_{\alpha+1}=\VV[X_\alpha\cup\{z\}]$ and $Y_{\alpha+1}=Y_{\alpha}\cup\{g_\alpha(z)\}$. Conclude the construction by setting $X=\bigcup_{\alpha\in\cccc}X_\alpha$.

Since $\VV[X]=X$ by condition (III), it follows from Lemma \ref{jan} that $X$ is homogeneous. It is clear that conditions (\ref{preserve}) and (\ref{pushaway}) of Theorem \ref{general} are satisfied. To see that condition (\ref{kill}) holds, assume that $g\in\GG$ and $|\{x\in\dom(g):g(x)\notin\Ss(x)\}|=\cccc$ for every subgroup $\Ss$ of $\VV$ such that $|\Ss|<\add$. Since $\VV$ is countable, this trivially implies that $|\{x\in\dom(g_\alpha):g_\alpha(x)\notin\VV(x)\}|=\cccc$, where $\alpha\in\cccc$ is such that $g_\alpha=g$. It follows from conditions (IV) and (II) that there exists $z\in\dom(g)\cap X$ such that $g(z)\notin X$.
\end{proof}

\begin{corollary}\label{hnotrh}
There exists a homogeneous subspace of $2^\omega$ that is not relatively homogeneous.
\end{corollary}

\section{The countable dense homogeneous example}

\begin{theorem}\label{cdh}
Assume $\MAsigma$. Then there exists a countable dense homogeneous subspace of $2^\omega$ whose complement is dense in $2^\omega$ and rigid.
\end{theorem}
\begin{proof}
Once again, we plan to apply Theorem \ref{general}. Let $\GG=\{g_\alpha:\alpha\in\cccc\}$ be an enumeration. Enumerate as $\{(A_\alpha,B_\alpha):\alpha\in\cccc\}$ all pairs of countable dense subsets of $2^\omega$, making sure to list each pair cofinally often. Fix a countable dense subset $D$ of $2^\omega$ such that $D\cap Q=\varnothing$.

Using transfinite recursion, we will construct an increasing sequence $\langle \HH_\alpha:\alpha\in\cccc\rangle$ consisting of subgroups of $\HH(2^\omega)$, together with increasing sequences $\langle X_\alpha:\alpha\in\cccc\rangle$ and $\langle Y_\alpha:\alpha\in\cccc\rangle$ consisting of subsets of $2^\omega$, so that the following conditions are satisfied for every $\alpha\in\cccc$.
\begin{enumerate}
\item[(I)] $|X_\alpha|<\cccc$ and $|Y_\alpha|<\cccc$.
\item[(II)] $X_\alpha\cap Y_\alpha=\varnothing$.
\item[(III)] $|\HH_\alpha|<\cccc$.
\item[(IV)] $\HH_\alpha[X_\alpha]=X_\alpha$.
\item[(V)] If $h\in\HH_\alpha\setminus\{\id\}$ then $h[D]\cap D=\varnothing$.
\item[(VI)] If $|\{x\in\dom(g_\alpha):g_\alpha(x)\notin\HH_\alpha(x)\}|=\cccc$ then there exists $z\in\dom(g_\alpha)\cap X_{\alpha+1}$ such that $g(z)\in Y_{\alpha+1}$.
\item[(VII)] If $A_\alpha\cup B_\alpha\subseteq X_\alpha$ then there exists $f\in\HH_{\alpha+1}$ such that $f[A_\alpha]=B_\alpha$.
\end{enumerate}
Start by setting $\HH_0=\{\id\}$, $X_0=\varnothing$ and $Y_0=D$. Take unions at limit stages. At a successor stage $\alpha+1$, suppose that $\HH_\alpha$, $X_\alpha$ and $Y_\alpha$ have already been constructed. Assume that $|\{x\in\dom(g_\alpha):g_\alpha(x)\notin\HH_\alpha(x)\}|=\cccc$. Then, it is possible to fix
$$
z\in\dom(g_\alpha)\setminus (\HH_\alpha[Y_\alpha]\cup g_\alpha^{-1}[X_\alpha])
$$
such that $g_\alpha(z)\notin\HH_\alpha(z)$. First we will construct $\HH_{\alpha+1}$. If $A_\alpha\cup B_\alpha\nsubseteq X_\alpha$, set $\HH_{\alpha+1}=\HH_\alpha$. If $A_\alpha\cup B_\alpha\subseteq X_\alpha$, let $f\in\HH(2^\omega)$ be obtained by applying Lemma \ref{keylemma} with $\HH=\HH_\alpha$, $X=X_\alpha\cup\{z\}$, $Y=Y_\alpha\cup\{g_\alpha(z)\}$, $A=A_\alpha$ and $B=B_\alpha$, then set $\HH_{\alpha+1}=\langle\HH_\alpha\cup\{f\}\rangle$. Finally, set $X_{\alpha+1}=\HH_{\alpha+1}[X_\alpha\cup\{z\}]$ and $Y_{\alpha+1}=Y_\alpha\cup\{g_\alpha(z)\}$. Conclude the construction by setting $X=\bigcup_{\alpha\in\cccc}X_\alpha$ and $\HH=\bigcup_{\alpha\in\cccc}\HH_\alpha$.

Notice that $\HH[X]=X$ by condition (IV). Using condition (VII), it is straightforward to check that $X$ is countable dense homogeneous. It is clear that conditions (\ref{preserve}) and (\ref{pushaway}) of Theorem \ref{general} are satisfied. To see that condition (\ref{kill}) holds, assume that $g\in\GG$ and $|\{x\in\dom(g):g(x)\notin\Ss(x)\}|=\cccc$ for every subgroup $\Ss$ of $\HH$ such that $|\Ss|<\add$. Fix $\alpha\in\cccc$ such that $g_\alpha=g$, and notice that $|\{x\in\dom(g_\alpha):g_\alpha(x)\notin\HH_\alpha(x)\}|=\cccc$ because $|\HH_\alpha|<\cccc=\add$ by condition (III) and $\MAsigma$. It follows from conditions (VI) and (II) that there exists $z\in\dom(g)\cap X$ such that $g(z)\notin X$.
\end{proof}

\begin{lemma}\label{keylemma}
Assume $\MAsigma$. Furthermore, assume that the following requirements are satisfied.
\begin{itemize}
\item $X$ and $Y$ are subsets of $2^\omega$ of size less than $\cccc$.
\item  $\HH$ is a subgroup of $\HH(2^\omega)$ of size less than $\cccc$.
\item $\HH[X]\cap Y=\varnothing$.
\item $D$ is a countable dense subset of $2^\omega$ such that $D\subseteq Y$.
\item $h[D]\cap D=\varnothing$ for every $h\in\HH\setminus\{\id\}$.
\item $A$ and $B$ are countable dense subsets of $2^\omega$ such that $A\cup B\subseteq X$.
\end{itemize}
\newpage
\noindent Then there exists $f\in\HH(2^\omega)$ such that the following conditions hold.
\begin{enumerate}
\item $f[A]=B$.
\item $\langle\HH\cup\{f\}\rangle[X]\cap Y=\varnothing$.
\item $h[D]\cap D=\varnothing$ for every $h\in\langle\HH\cup\{f\}\rangle\setminus\{\id\}$.
\end{enumerate}
\end{lemma}
\begin{proof}
Let $\lambda=\max\{|X|,|\HH|,\omega\}$, and notice that $\lambda<\cccc$. Start by constructing $X^\ast\supseteq X$ so that $X^\ast\setminus (A\cup B)$ is $\lambda$-dense in $2^\omega$ and $\HH[X^\ast]\cap Y=\varnothing$. Notice that $\HH[X^\ast]\setminus (A\cup B)$ will still be $\lambda$-dense in $2^\omega$. We claim that any $f\in\HH(2^\omega)$ that satisfies the following conditions will also satisfy conditions (1) and (2).
\begin{enumerate}
\item[(I)] $f[A]=B$.
\item[(II)] $f[\HH[X^\ast]\setminus A]=f[\HH[X^\ast]\setminus B]$.
\end{enumerate}
In fact, conditions (I) and (II) immediately imply that $f[\HH[X^\ast]]=\HH[X^\ast]$, and therefore
$$
\langle\HH\cup\{f\}\rangle[X]\subseteq\langle\HH\cup\{f\}\rangle[\HH[X^\ast]]=\HH[X^\ast],
$$
which is disjoint from $Y$ by construction.

Let $\HH^\ast=\HH\setminus\{\id\}$. Define $T_h=\{x\in 2^\omega:h(x)=x\}$ for $h\in\HH$, and observe that each $T_h$ is closed. Furthermore, if $h\in\HH^\ast$ then $T_h$ is nowhere dense, since $D\subseteq 2^\omega\setminus T_h$. Define
$$
C=\{x\in 2^\omega:h(x)\neq x\text{ for every }h\in\HH^\ast\}.
$$
Since $|\HH^\ast|<\cccc=\add$ by $\MAsigma$, one sees that $C=2^\omega\setminus\bigcup_{h\in\HH^\ast}T_h$ is comeager in $2^\omega$. In particular, $C$ is $\cccc$-dense in $2^\omega$, so it is possible to fix a collection $\DD$ of size $\lambda$ consisting of countable dense subsets of $2^\omega$ such that the following conditions hold.
\begin{itemize}
\item $\bigcup\DD\subseteq C\setminus\HH[X^\ast\cup Y]$.
\item $E\cap F=\varnothing$ whenever $E,F\in\DD$ and $E\neq F$.
\item $\HH(x)\cap\HH(y)=\varnothing$ whenever $x,y\in\bigcup\DD$ and $x\neq y$.
\end{itemize}
Given arbitrary $E,F\in\DD$ and $h,g\in\HH$, one can easily verify that $h[E]\cap g[F]=\varnothing$ unless $E=F$ and $h=g$.

Let $\TT$ be the set of all $s\in (\HH^\ast\cup\{-1,1\})^{<\omega}$ such that the following conditions hold whenever $\{k,k+1\}\subseteq\dom(s)$. For notational convenience, we will always assume that $s=\langle s_{n-1},\ldots,s_0\rangle$, where $n=\dom(s)$.
\begin{itemize}
\item If $s_k\in\HH^\ast$ then $s_{k+1}\in\{-1,1\}$.
\item If $s_k\in\{-1,1\}$ then $s_{k+1}\in\{s_k\}\cup\HH^\ast$.
\end{itemize}
Notice that $|\TT|\leq\lambda<\cccc$. Suppose that some $f\in\HH(2^\omega)$ has been chosen. Given $s\in\TT$ such that $\dom(s)=n$, we will say that $w\in\langle\HH\cup\{f\}\rangle$ is of \emph{type} $s$ if it can be written as
$$
w=h_{n-1}\cdots h_1h_0,
$$
where $h_k=s_k$ if $s_k\in\HH^\ast$ and $h_k=f^{s_k}$ if $s_k\in\{-1,1\}$. In particular, $\id$ is the only element of type $\varnothing$.

Define $D_s$ for $s\in\TT$ so that the following conditions hold.
\begin{itemize}
\item $D_\varnothing=D$.
\item $D_{\langle h\rangle^\frown s}=h[D_s]$ whenever $h\in\HH^\ast$ and $\langle h\rangle^\frown s\in\TT$.
\item $\DD=\{D_{\langle 1\rangle^\frown s}:s\in\TT\}\cup\{D_{\langle -1\rangle^\frown s}:s\in\TT\}$ is an injective enumeration.
\end{itemize}
Using the properties of $\DD$, it is straightforward to verify that $D_s\cap D_t=\varnothing$ whenever $s,t\in\TT$ and $s\neq t$.
Therefore, by Lemma \ref{bb}, we can fix $f\in\HH(2^\omega)$ that satisfies the following conditions, together with conditions (I) and (II).
\begin{enumerate}
\item[(III)] $f[D_s]=D_{\langle 1\rangle^\frown s}$ whenever $\langle 1\rangle^\frown s\in\TT$.
\item[(IV)] $f[D_{\langle -1\rangle^\frown s}]=D_s$  whenever $\langle -1\rangle^\frown s\in\TT$.
\end{enumerate}

In order to see that condition (3) holds, we will use induction on $n=\dom(s)$ to prove that $w[D]=D_s$ whenever $w\in\langle\HH\cup\{f\}\rangle$ is of type $s$. The case $n=0$ is trivial, so assume that $n>0$. First assume that $s=\langle h\rangle^\frown s'$ for some $h\in\HH^\ast$ and $s'\in\TT$. This means that $w=hw'$ for some $w'$ of type $s'$. Using the inductive assumption $w'[D]=D_{s'}$, one sees that
$$
w[D]=h[w'[D]]=h[D_{s'}]=D_{\langle h\rangle^\frown s'}=D_s.
$$
Now assume that $s=\langle 1\rangle^\frown s'$ for some $s'\in\TT$. This means that $w=fw'$ for some $w'$ of type $s'$. Using condition (III) and the inductive assumption $w'[D]=D_{s'}$, one sees that
$$
w[D]=f[w'[D]]=f[D_{s'}]=D_{\langle 1\rangle^\frown s'}=D_s.
$$
The case in which $s=\langle -1\rangle^\frown s'$ for some $s'\in\TT$ can be dealt with similarly, using the fact that $f^{-1}[D_{s'}]=D_{\langle -1\rangle^\frown s'}$ by condition (IV).
\end{proof}

\begin{corollary}\label{cdhnotrcdh}
Assume $\MAsigma$. Then there exists a countable dense homogeneous subspace of $2^\omega$ that is not relatively countable dense homogeneous.
\end{corollary}

\begin{question}
Is it possible to prove in $\ZFC$ that there exists a countable dense homogeneous subspace of $2^\omega$ whose complement is dense in $2^\omega$ and rigid?
\end{question}

\end{document}